\documentclass[3p,11pt,a4paper,twoside,fleqn,sort&compress]{article}
\usepackage{amsfonts}
\usepackage{amsmath}
\usepackage{amsthm}
\usepackage{hyperref}
\usepackage{amssymb,amsthm}
\usepackage[mathscr]{euscript}
\usepackage{amssymb,amsmath,latexsym}
\usepackage[bottom]{footmisc}
\usepackage[affil-it]{authblk}
\usepackage{mathtools}
\usepackage{cite}

\newtheorem{theorem}{Theorem}[section]

\newtheorem{definition}[theorem]{Definition}
\newtheorem{example}[theorem]{Example}

\newtheorem{lemma}[theorem]{Lemma}

\numberwithin{equation}{section}
\begin{document}

\title{Four-point boundary value problems for a coupled system of fractional differential equations with $\psi$-Caputo fractional derivatives}
\author{Mohamed I. Abbas\thanks{miabbas77@gmail.com}\\ Department of Mathematics and Computer Science, Faculty of Science, Alexandria University, Alexandria 21511, Egypt}
\date{}

\maketitle
\makeatletter
\renewcommand\@makefnmark%
{\mbox{\textsuperscript{\normalfont\@thefnmark)}}}
\makeatother

\begin{abstract}
In this paper, we focus on  the existence and uniqueness of  solutions of boundary value problems for a coupled system of fractional differential equations with four-point boundary conditions involving $\psi$-Caputo fractional derivatives. Our investigation is based on Leray-Schauder alternative and  Banach's fixed point theorem. Two examples are presented to illustrate the applicability of the results developed.
\end{abstract}

\section{Introduction}
Fractional differential equations have been studied extensively in the literature because of their applications in various fields of engineering and science, for example, see the monographs \cite{Kilbas,Lakshmikantham,Miller,Podlubny,Samko}. The study of qualitative properties of solutions for different fractional differential equations such as existence, uniqueness, stability, continuous dependence, etc., has become an important area of investigation in recent years. For instance, we mention here few recent works by Abbas, M.I. \cite{Abbas0} -\cite{Abbas5},  also see the references cited therein.\\

 The topic of boundary value problems for fractional differential equations has also received considerable attention during recent decades.  At the same time, it has become widely seen that coupled boundary value problems have gained an importance role in view of their great utility in handling of applied nature such as: ecological models \cite{Javidi}, anomalous diffusion \cite{diffusion}, systems of nonlocal thermoelasticity \cite{thermo},  the heat equations \cite{heat}, etc. For some recent theoretical results on the
topic, we refer the reader to a series of papers \cite{Agarwal,Ahmad18,Ahmad16,RaoC,Zou} and the references cited therein.\\

Very recently, Almeida \cite{Almeida1}, introduced a Caputo type fractional derivative with respect to another function called by $\psi$-Caputo fractional operator. He also studied some properties like the semigroup law, a relationship between the fractional derivative and the fractional integral, Taylor's theorem, Fermat's theorem, etc. For recent works involving $\psi$-Caputo fractional operator, we refer the reader to a series of papers \cite{Abdo,Almeida,Belmor,Derbazi,Vivek}.\\

In this paper, we establish the existence and uniqueness results of the following coupled system of fractional differential equations:
\begin{equation}\label{main}
\begin{cases}
{}^{C}D_{0^{+}}^{\alpha;\psi}x(t)=f(t,x(t),y(t)),~~t\in [0,1],~1<\alpha<2\\
{}^{C}D_{0^{+}}^{\beta;\psi}y(t)=g(t,x(t),y(t)),~~t\in [0,1],~1<\beta<2\\
x(0)=y(0)=0,\\
x(1)=\lambda x(\eta),~y(1)=\mu y(\xi),~0<\eta,\xi<1,~\lambda,\mu>0\\
\end{cases}
\end{equation}
where ${}^{C}D_{0^{+}}^{\alpha;\psi},{}^{C}D_{0^{+}}^{\beta;\psi}$ denote the $\psi$-Caputo fractional derivatives of order $\alpha,\beta$ and $f,g:[0,1]\times\mathbb{R}^{2}\to \mathbb{R}$ are given continuous functions.\\

This paper is organized as follows. In Section 2, we give some definitions and lemmas used in this paper. Existence and uniqueness results for (\ref{main}) are derived in Sections 3. Two examples are provided in Section 4.

\section{Preliminary and Lemmas}
In this section, we recall definitions, properties and lemmas of the new $\psi$-Caputo fractional derivative. 

\begin{definition}(\cite{Almeida1,Almeida})
For $\alpha>0$, the left-sided $\psi$-Riemann-Liouville fractional integral of order $\alpha$ for an integrable function $\sigma:[a,b]\to\mathbb{R}$ with respect to another function $\psi:[a,b]\to\mathbb{R}$ that is an increasing differentiable function such that $\psi^{'}(t)\not=0$, for all $t\in [a,b]$ is defined by
\begin{equation}
I_{a^{+}}^{\alpha;\psi}\sigma(t)=\frac{1}{\Gamma(\alpha)}\int_{a}^{t}\psi^{'}(s)(\psi(t)-\psi(s))^{\alpha-1}\sigma(s)ds,
\end{equation}
where $\Gamma$ is the Euler Gamma function.
\end{definition}

\begin{definition}(\cite{Almeida1})
Let $n\in\mathbb{N}$ and let $\psi,\sigma\in C^{n}([a,b],\mathbb{R})$ be two functions such that $\psi$ is increasing and $\psi^{'}(t)\not=0$, for all $t\in [a,b]$. The left-sided $\psi$-Riemann-Liouville fractional derivative of a function $\sigma$ of order $\alpha$ is defined by
\begin{eqnarray*}
D_{a^{+}}^{\alpha;\psi}\sigma(t)&=&\bigg(\frac{1}{\psi^{'}(t)}\frac{d}{dt}\bigg)^{n}I_{a^{+}}^{n-\alpha;\psi}\sigma(t)\\
&=&\frac{1}{\Gamma(n-\alpha)}\bigg(\frac{1}{\psi^{'}(t)}\frac{d}{dt}\bigg)^{n}\int_{a}^{t}\psi^{'}(s)(\psi(t)-\psi(s))^{n-\alpha-1}\sigma(s)ds,
\end{eqnarray*}
where $n=[\alpha]+1$ and $[\alpha]$ denotes the integer part of the real number $\alpha$.
\end{definition}

\begin{definition}(\cite{Almeida1})
Let $n-1<\alpha<n,~n\in\mathbb{N}$ and let $\psi,\sigma\in C^{n}([a,b],\mathbb{R})$ be two functions such that $\psi$ is increasing and $\psi^{'}(t)\not=0$, for all $t\in [a,b]$. The left-sided $\psi$-Caputo fractional derivative of a function $\sigma$ of order $\alpha$ is defined by
\begin{eqnarray*}
{}^{C}D_{a^{+}}^{\alpha;\psi}\sigma(t)&=&D_{a^{+}}^{\alpha;\psi}\bigg[\sigma(t)-\sum_{k=0}^{n-1}\frac{\sigma_{\psi}^{[k]}(a)}{k!}(\psi(t)-\psi(a))^{k}\bigg],
\end{eqnarray*}
where $\sigma_{\psi}^{[k]}(t)=\bigg(\frac{1}{\psi^{'}(t)}\frac{d}{dt}\bigg)^{k}\sigma(t)$ and $n=[\alpha]+1$ for $\alpha\not\in\mathbb{N}$, $n=\alpha$ for $\alpha\in\mathbb{N}$. Further, if $\sigma\in C^{n}([a,b],\mathbb{R})$ and $\alpha\not\in\mathbb{N}$,then

\begin{eqnarray*}
{}^{C}D_{a^{+}}^{\alpha;\psi}\sigma(t)&=&I_{a^{+}}^{n-\alpha;\psi}\bigg(\frac{1}{\psi^{'}(t)}\frac{d}{dt}\bigg)^{n}\sigma(t)\\
&=&\frac{1}{\Gamma(n-\alpha)}\int_{a}^{t}\psi^{'}(s)(\psi(t)-\psi(s))^{n-\alpha-1}\sigma_{\psi}^{[n]}(s)ds.
\end{eqnarray*}
Thus, if $\alpha=n\in\mathbb{N}$, then ${}^{C}D_{a^{+}}^{\alpha;\psi}\sigma(t)=\sigma_{\psi}^{[n]}(t)$.

\end{definition}

\begin{lemma}(\cite{Almeida})\label{lemID}
Let $\alpha>0$. The following holds:\\
If $\sigma\in C([a,b],\mathbb{R})$, then $${}^{C}D_{a^{+}}^{\alpha;\psi}I_{a^{+}}^{\alpha;\psi}\sigma(t)=\sigma(t),~~t\in[a,b].$$
If $\sigma\in C^{n}([a,b],\mathbb{R}),~n-1<\alpha<n$, then
$$I_{a^{+}}^{\alpha;\psi}~{}^{C}D_{a^{+}}^{\alpha;\psi}\sigma(t)=\sigma(t)-\sum_{k=0}^{n-1}c_{k}(\psi(t)-\psi(a))^{k},~~t\in[a,b],$$
where $c_{k}=\frac{\sigma_{\psi}^{[k]}(a)}{k!}$.
\end{lemma}

\begin{lemma}(\cite{Almeida,Kilbas})\label{lemprop}
Let $t>a$, $\alpha\geq 0$ and $\beta>0$. Then
\begin{itemize}
  \item $I_{a^{+}}^{\alpha;\psi}(\psi(t)-\psi(a))^{\beta-1}=\frac{\Gamma(\beta)}{\Gamma(\beta+\alpha)}(\psi(t)-\psi(a))^{\beta+\alpha-1}$,
  \item ${}^{C}D_{a^{+}}^{\alpha;\psi}(\psi(t)-\psi(a))^{\beta-1}=\frac{\Gamma(\beta)}{\Gamma(\beta-\alpha)}(\psi(t)-\psi(a))^{\beta-\alpha-1}$,
  \item ${}^{C}D_{a^{+}}^{\alpha;\psi}(\psi(t)-\psi(a))^{k}=0,~\textnormal{for all}~k\in\{0,1,\cdots,n-1\},~n\in\mathbb{N}$.
  \end{itemize}
\end{lemma}

\begin{lemma}(Leray-Schauder alternative \cite{Leray})\label{fixed}
Let $\mathcal{T}: E\to E$ be a completely continuous operator (i.e., a map that restricted to any bounded
set in E is compact). Let $\mathcal{S}(\mathcal{T}) = \{ x\in E:x=\nu\mathcal{T}(x),~\text{for some}~ 0 < \nu< 1 \}$. Then either the set $\mathcal{S}(\mathcal{T})$ is unbounded or $\mathcal{T}$ has at least one fixed point.
\end{lemma}

\section{Main Results}
Let $C([0,1],\mathbb{R})$ be the space of all continuous functions defined on $[0,1]$. Let $X=\{x(t):x(t)\in C([0,1],\mathbb{R})\}$ be a Banach space endowed with the norm $\|x\|_{X}=\sup_{t\in[0,1]}|x(t)|$ and $Y=\{y(t):y(t)\in C([0,1],\mathbb{R})\}$ be a Banach space endowed with the norm $\|y\|_{Y}=\sup_{t\in[0,1]}|y(t)|$. Thus the product space $(X\times Y,\|\cdot\|_{X\times Y})$ is also a Banach space with the norm $\|(x,y)\|_{X\times Y}=\|x\|_{X}+\|y\|_{Y}$ for $(x,y)\in X\times Y$.\\

\begin{lemma}
Let $h\in C([0,1],,\mathbb{R})$ be a given function and $1 <\alpha< 2$. Then the unique solution of
\begin{equation}\label{linear}
\begin{cases}
{}^{C}D_{0^{+}}^{\alpha;\psi}x(t)=h(t),~~t\in[0,1]\\
x(0)=0,~x(1)=\lambda x(\eta),
\end{cases}
\end{equation}
is given by the integral equation
\begin{eqnarray}\label{equiv}
x(t)&=&\frac{1}{\Gamma(\alpha)}\int_{0}^{t}\psi^{'}(s)(\psi(t)-\psi(s))^{\alpha-1}h(s)ds+\frac{1}{\Delta_1}\bigg[\frac{1}{\Gamma(\alpha)}\int_{0}^{1}\psi^{'}(s)(\psi(1)-\psi(s))^{\alpha-1}h(s)ds\nonumber\\
&&-\frac{\lambda}{\Gamma(\alpha)}\int_{0}^{\eta}\psi^{'}(s)(\psi(\eta)-\psi(s))^{\alpha-1}h(s)ds\bigg](\psi(t)-\psi(0)),
\end{eqnarray}
where $$\Delta_1=\lambda(\psi(\eta)-\psi(0))-(\psi(1)-\psi(0))\not=0.$$
\end{lemma}
\begin{proof}
First, let $x\in C([0,1])$ be a solution of (\ref{linear}). Then, by Lemma \ref{lemID}, we get
\begin{equation}\label{first}
x(t)=\frac{1}{\Gamma(\alpha)}\int_{0}^{t}\psi^{'}(s)(\psi(t)-\psi(s))^{\alpha-1}h(s)ds+c_0+c_1(\psi(t)-\psi(0)).
\end{equation}
Applying the boundary conditions $x(0)=0,~x(1)=\lambda x(\eta)$ implies that $c_0=0$ and 
\begin{multline*}
\frac{1}{\Gamma(\alpha)}\int_{0}^{1}\psi^{'}(s)(\psi(1)-\psi(s))^{\alpha-1}h(s)ds+c_1(\psi(1)-\psi(0))\\
=\frac{\lambda}{\Gamma(\alpha)}\int_{0}^{\eta}\psi^{'}(s)(\psi(\eta)-\psi(s))^{\alpha-1}h(s)ds
+\lambda c_1(\psi(\eta)-\psi(0)),
\end{multline*}
which implies that 
$$c_1=\frac{1}{\Delta_1}\bigg[\frac{1}{\Gamma(\alpha)}\int_{0}^{1}\psi^{'}(s)(\psi(1)-\psi(s))^{\alpha-1}h(s)ds-\frac{\lambda}{\Gamma(\alpha)}\int_{0}^{\eta}\psi^{'}(s)(\psi(\eta)-\psi(s))^{\alpha-1}h(s)ds\bigg].$$
which, on substituting in (\ref{first}), completes the solution (\ref{equiv}).\\
Conversely, If $x(t)$ satisfies the integral equation (\ref{equiv}), then by applying the $\psi$-Caputo fractional derivative ${}^{C}D_{0^{+}}^{\alpha;\psi}$ to both sides of equation (\ref{equiv}) and using Lemmas \ref{lemID} and \ref{lemprop}, we obtain 
$${}^{C}D_{0^{+}}^{\alpha;\psi}x(t)=h(t).$$
Finally, it remains to show that the boundary conditions in (\ref{linear}) are satisfied. Clearly, $x(0)=0$ and the direct computations lead to $x(1)=\lambda x(\eta)$. This completes the proof.
\end{proof}
Similarly, the general solution of ${}^{C}D_{0^{+}}^{\beta;\psi}y(t)=h(t),~y(0)=0,~y(1)=\mu y(\xi)$  can be obtained from 
\begin{eqnarray}\label{equiv2}
y(t)&=&\frac{1}{\Gamma(\beta)}\int_{0}^{t}\psi^{'}(s)(\psi(t)-\psi(s))^{\beta-1}h(s)ds+\frac{1}{\Delta_2}\bigg[\frac{1}{\Gamma(\beta)}\int_{0}^{1}\psi^{'}(s)(\psi(1)-\psi(s))^{\beta-1}h(s)ds\nonumber\\
&&-\frac{\mu}{\Gamma(\beta)}\int_{0}^{\xi}\psi^{'}(s)(\psi(\xi)-\psi(s))^{\beta-1}h(s)ds\bigg](\psi(t)-\psi(0)),
\end{eqnarray}
where $$\Delta_2=\mu(\psi(\xi)-\psi(0))-(\psi(1)-\psi(0))\not=0.$$
\begin{lemma}
Assume that $f,g:[0,1]\times\mathbb{R}^{2}\to\mathbb{R}$ are continuous functions. Then  $(x,y)\in X\times Y$ is a solution of (\ref{main}) if and only if $(x,y)\in X\times Y$ is a solution of the coupled system of integral equations
\begin{eqnarray*}
x(t)&=&\frac{1}{\Gamma(\alpha)}\int_{0}^{t}\psi^{'}(s)(\psi(t)-\psi(s))^{\alpha-1}f(s,x(s),y(s))ds\\
&+&\frac{1}{\Delta_1}\bigg[\frac{1}{\Gamma(\alpha)}\int_{0}^{1}\psi^{'}(s)(\psi(1)-\psi(s))^{\alpha-1}f(s,x(s),y(s))ds\\
&-&\frac{\lambda}{\Gamma(\alpha)}\int_{0}^{\eta}\psi^{'}(s)(\psi(\eta)-\psi(s))^{\alpha-1}f(s,x(s),y(s))ds\bigg](\psi(t)-\psi(0)),\\
y(t)&=&\frac{1}{\Gamma(\beta)}\int_{0}^{t}\psi^{'}(s)(\psi(t)-\psi(s))^{\beta-1}g(s,x(s),y(s))ds\\
&+&\frac{1}{\Delta_2}\bigg[\frac{1}{\Gamma(\beta)}\int_{0}^{1}\psi^{'}(s)(\psi(1)-\psi(s))^{\beta-1}g(s,x(s),y(s))ds\\
&-&\frac{\mu}{\Gamma(\beta)}\int_{0}^{\xi}\psi^{'}(s)(\psi(\xi)-\psi(s))^{\beta-1}g(s,x(s),y(s))ds\bigg](\psi(t)-\psi(0)).
\end{eqnarray*}
\end{lemma}

Let us define the operator $\mathcal{T}:X\times Y\to X\times Y$ as
\begin{equation} \label{op}
\mathcal{T}(x,y)(t)=
\begin{pmatrix}
\mathcal{T}_{1}(x,y)(t)\\
\mathcal{T}_{2}(x,y)(t)
\end{pmatrix}
\end{equation}
where
\begin{equation} \label{op1}
\begin{split}
\mathcal{T}_{1}(x,y)(t)&=\frac{1}{\Gamma(\alpha)}\int_{0}^{t}\psi^{'}(s)(\psi(t)-\psi(s))^{\alpha-1}f(s,x(s),y(s))ds\\
&+\frac{1}{\Delta_1}\bigg[\frac{1}{\Gamma(\alpha)}\int_{0}^{1}\psi^{'}(s)(\psi(1)-\psi(s))^{\alpha-1}f(s,x(s),y(s))ds\\
&-\frac{\lambda}{\Gamma(\alpha)}\int_{0}^{\eta}\psi^{'}(s)(\psi(\eta)-\psi(s))^{\alpha-1}f(s,x(s),y(s))ds\bigg](\psi(t)-\psi(0)),
\end{split}
\end{equation}
and
\begin{equation} \label{op2}
\begin{split}
\mathcal{T}_{2}(x,y)(t)&=\frac{1}{\Gamma(\beta)}\int_{0}^{t}\psi^{'}(s)(\psi(t)-\psi(s))^{\beta-1}g(s,x(s),y(s))ds\\
&+\frac{1}{\Delta_2}\bigg[\frac{1}{\Gamma(\beta)}\int_{0}^{1}\psi^{'}(s)(\psi(1)-\psi(s))^{\beta-1}g(s,x(s),y(s))ds\\
&-\frac{\mu}{\Gamma(\beta)}\int_{0}^{\xi}\psi^{'}(s)(\psi(\xi)-\psi(s))^{\beta-1}g(s,x(s),y(s))ds\bigg](\psi(t)-\psi(0)).
\end{split}
\end{equation}

In order to establish our main results, we introduce the following assumptions.
\begin{itemize}
  \item[](A1) The functions $~f,g:[0,1]\times\mathbb{R}^{2}\to\mathbb{R}$ are continuous and there exist real constants $L_1,L_2>0$ such that
  $$|f(t,x_1,y_1)-f(t,x_2,y_2)|\leq L_1(|x_1-x_2|+|y_1-y_2|),$$
  $$|g(t,x_1,y_1)-g(t,x_2,y_2)|\leq L_2(|x_1-x_2|+|y_1-y_2|),$$
$\forall t\in[0,1]~ \textnormal{and}~ x_i,y_i\in\mathbb{R},~i=1,2.$
  \item [](A2) There exist real constants $k_i,l_i\geq 0,~i=1,2$ and $k_0>0,l_0>0$ such that 
  $$|f(t,x,y)|\leq k_0+k_1|x|+k_2|y|,~~|g(t,x,y)|\leq l_0+l_1|x|+l_2|y|,$$
$\forall t\in[0,1]~ \textnormal{and}~ x_i,y_i\in\mathbb{R},~i=1,2.$
\end{itemize}
In the following, for brevity, we use the notations:
\begin{equation}\label{constants1}
 \left.\begin{aligned}
\gamma_1&=\frac{M_1}{\Gamma(\alpha+1)}\Bigg[(\psi(1)-\psi(0))^{\alpha}+\frac{(|\lambda|+1)}{|\Delta_1|}(\psi(1)-\psi(0))^{\alpha+1}\Bigg]\\
\gamma_2&=\frac{M_2}{\Gamma(\beta+1)}\Bigg[(\psi(1)-\psi(0))^{\beta}+\frac{(|\mu|+1)}{|\Delta_2|}(\psi(1)-\psi(0))^{\beta+1}\Bigg]\\
\gamma_3&=\frac{L_1}{\Gamma(\alpha+1)}\Bigg[(\psi(1)-\psi(0))^{\alpha}+\frac{(|\lambda|+1)}{|\Delta_1|}(\psi(1)-\psi(0))^{\alpha+1}\Bigg]\\
\gamma_4&=\frac{L_2}{\Gamma(\beta+1)}\Bigg[(\psi(1)-\psi(0))^{\beta}+\frac{(|\mu|+1)}{|\Delta_2|}(\psi(1)-\psi(0))^{\beta+1}\Bigg]
\end{aligned}\right\}
\end{equation}

\begin{equation}\label{constants2}
 \left.\begin{aligned}
\Omega_0&=\Bigg(\frac{1}{\Gamma(\alpha+1)}\bigg[(\psi(1)-\psi(0))^{\alpha}+\frac{(|\lambda|+1)}{|\Delta_1|}(\psi(1)-\psi(0))^{\alpha+1}\bigg]\\
&+\frac{1}{\Gamma(\beta+1)}\bigg[(\psi(1)-\psi(0))^{\beta}+\frac{(|\mu|+1)}{|\Delta_2|}(\psi(1)-\psi(0))^{\beta+1}\bigg]\Bigg)l_0\\
\Omega_1&=\frac{1}{\Gamma(\alpha+1)}\Bigg[(\psi(1)-\psi(0))^{\alpha}+\frac{(|\lambda|+1)}{|\Delta_1|}(\psi(1)-\psi(0))^{\alpha+1}\Bigg]k_1\\
&+\frac{1}{\Gamma(\beta+1)}\Bigg[(\psi(1)-\psi(0))^{\beta}+\frac{(|\mu|+1)}{|\Delta_2|}(\psi(1)-\psi(0))^{\beta+1}\Bigg]l_1\\
\Omega_2&=\frac{1}{\Gamma(\alpha+1)}\Bigg[(\psi(1)-\psi(0))^{\alpha}+\frac{(|\lambda|+1)}{|\Delta_1|}(\psi(1)-\psi(0))^{\alpha+1}\Bigg]k_2\\
&+\frac{1}{\Gamma(\beta+1)}\Bigg[(\psi(1)-\psi(0))^{\beta}+\frac{(|\mu|+1)}{|\Delta_2|}(\psi(1)-\psi(0))^{\beta+1}\Bigg]l_2\\
\Omega^{*}&=\max\{\Omega_1,\Omega_2\}
\end{aligned}\right\}
\end{equation}
\subsection{The uniqueness result via Banach's fixed point theorem}
\begin{theorem}\label{uniqueness}
Assume that (A1) hold. Then the coupled system (\ref{main}) has a unique solution on $[0,1]$ provided that
\begin{equation}\label{uniquecondition}
\left(\gamma_3+\gamma_4\right)<1,
\end{equation}
where $\gamma_3$ and $\gamma_4$ are given in (\ref{constants1}).
\end{theorem}
\begin{proof}
Assume that $r>0$ is a real number satisfying
$$r\geq\frac{\gamma_1+\gamma_2}{1-(\gamma_3+\gamma_4)},$$

 First we shall show that $\mathcal{T}\mathcal{B}_{r}\subset\mathcal{B}_{r}$, where $\mathcal{T}$ is given by (\ref{op}) and
$\mathcal{B}_{r}=\{(x,y)\in X\times Y:\|(x,y)\|_{X\times Y}\leq r\}$.\\

Set $\sup_{t\in[0,1]}|f(t,0,0)|=M_1<\infty$ and $\sup_{t\in[0,1]}|g(t,0,0)|=M_2<\infty$.
For $(x,y)\in\mathcal{B}_{r},~t\in[0,1]$, we have\\
$|\mathcal{T}_{1}(x,y)(t)|$
\begin{align*}
& \leq \Bigg|\frac{1}{\Gamma(\alpha)}\int_{0}^{t}\psi^{'}(s)(\psi(t)-\psi(s))^{\alpha-1}f(s,x(s),y(s))ds\Bigg|\\
&+\Bigg|\frac{1}{\Delta_1}\bigg[\frac{1}{\Gamma(\alpha)}\int_{0}^{1}\psi^{'}(s)(\psi(1)-\psi(s))^{\alpha-1}f(s,x(s),y(s))ds\\
&-\frac{\lambda}{\Gamma(\alpha)}\int_{0}^{\eta}\psi^{'}(s)(\psi(\eta)-\psi(s))^{\alpha-1}f(s,x(s),y(s))ds\bigg](\psi(t)-\psi(0))\Bigg|\\
& \leq \frac{1}{\Gamma(\alpha)}\int_{0}^{t}\psi^{'}(s)(\psi(t)-\psi(s))^{\alpha-1}(|f(s,x(s),y(s))-f(t,0,0)|+|f(t,0,0)|)ds\\
&+\frac{1}{|\Delta_1|}\bigg[\frac{1}{\Gamma(\alpha)}\int_{0}^{1}\psi^{'}(s)(\psi(1)-\psi(s))^{\alpha-1}(|f(s,x(s),y(s))-f(t,0,0)|+|f(t,0,0)|)ds\\
&+\frac{|\lambda|}{\Gamma(\alpha)}\int_{0}^{\eta}\psi^{'}(s)(\psi(\eta)-\psi(s))^{\alpha-1}(|f(s,x(s),y(s))-f(t,0,0)|+|f(t,0,0)|)ds\bigg]|\psi(t)-\psi(0)|\\
& \leq \frac{1}{\Gamma(\alpha)}\int_{0}^{t}\psi^{'}(s)(\psi(t)-\psi(s))^{\alpha-1}(L_1(|x(s)|+|y(s)|)+|f(t,0,0)|)ds\\
&+\frac{1}{|\Delta_1|}\bigg[\frac{1}{\Gamma(\alpha)}\int_{0}^{1}\psi^{'}(s)(\psi(1)-\psi(s))^{\alpha-1}(L_1(|x(s)|+|y(s)|)+|f(t,0,0)|)ds\\
&+\frac{|\lambda|}{\Gamma(\alpha)}\int_{0}^{\eta}\psi^{'}(s)(\psi(\eta)-\psi(s))^{\alpha-1}(L_1(|x(s)|+|y(s)|)+|f(t,0,0)|)ds\bigg]|\psi(t)-\psi(0)|\\
& \leq \frac{1}{\Gamma(\alpha)}\int_{0}^{t}\psi^{'}(s)(\psi(t)-\psi(s))^{\alpha-1}(L_1(\|x\|_X+\|y\|_Y)+M_1)ds\\
&+\frac{1}{|\Delta_1|}\bigg[\frac{1}{\Gamma(\alpha)}\int_{0}^{1}\psi^{'}(s)(\psi(1)-\psi(s))^{\alpha-1}(L_1(\|x\|_X+\|y\|_Y)+M_1)ds\\
&+\frac{|\lambda|}{\Gamma(\alpha)}\int_{0}^{\eta}\psi^{'}(s)(\psi(\eta)-\psi(s))^{\alpha-1}(L_1(\|x\|_X+\|y\|_Y)+M_1)ds\bigg](\psi(1)-\psi(0))\\
& \leq \frac{1}{\Gamma(\alpha+1)}\Bigg[(\psi(1)-\psi(0))^{\alpha}+\frac{(|\lambda|+1)}{|\Delta_1|}(\psi(1)-\psi(0))^{\alpha+1}\Bigg](L_1r+M_1),
\end{align*}
which implies that  
\begin{equation}\label{uniquT1}
\|\mathcal{T}_{1}(x,y)\|_{X}\leq \frac{1}{\Gamma(\alpha+1)}\Bigg[(\psi(1)-\psi(0))^{\alpha}+\frac{(|\lambda|+1)}{|\Delta_1|}(\psi(1)-\psi(0))^{\alpha+1}\Bigg](L_1r+M_1).
\end{equation}
Similarly, we can find that
\begin{equation}\label{uniquT2}
\|\mathcal{T}_{2}(x,y)\|_{Y}\leq \frac{1}{\Gamma(\beta+1)}\Bigg[(\psi(1)-\psi(0))^{\beta}+\frac{(|\mu|+1)}{|\Delta_2|}(\psi(1)-\psi(0))^{\beta+1}\Bigg](L_2r+M_2).
\end{equation}
Consequently, from (\ref{uniquT1}) and (\ref{uniquT2}), we get
\begin{align*}
\|\mathcal{T}(x,y)\|_{X\times Y}&\leq \frac{1}{\Gamma(\alpha+1)}\Bigg[(\psi(1)-\psi(0))^{\alpha}+\frac{(|\lambda|+1)}{|\Delta_1|}(\psi(1)-\psi(0))^{\alpha+1}\Bigg](L_1r+M_1)\\
&+\frac{1}{\Gamma(\beta+1)}\Bigg[(\psi(1)-\psi(0))^{\beta}+\frac{(|\mu|+1)}{|\Delta_2|}(\psi(1)-\psi(0))^{\beta+1}\Bigg](L_2r+M_2)\\
&=(\gamma_1+\gamma_2)+(\gamma_3+\gamma_4)r\\
&\leq r.
\end{align*}

Hence, $\mathcal{T}\mathcal{B}_{r}\subset\mathcal{B}_{r}$.\\

Now, for $(x_1, y_1), (x_2, y_2) \in X \times X$ and for any $t \in [0,1]$, we get\\

$|\mathcal{T}_{1}(x_1,y_1)(t)-\mathcal{T}_{1}(x_2,y_2)(t)|$
\begin{align*}
& \leq \frac{1}{\Gamma(\alpha)}\int_{0}^{t}\psi^{'}(s)(\psi(t)-\psi(s))^{\alpha-1}|f(s,x_1(s),y_1(s))-f(t,x_2(s),y_2(s))|ds\\
&+\frac{1}{|\Delta_1|}\bigg[\frac{1}{\Gamma(\alpha)}\int_{0}^{1}\psi^{'}(s)(\psi(1)-\psi(s))^{\alpha-1}|f(s,x_1(s),y_1(s))-f(t,x_2(s),y_2(s))|ds\\
&+\frac{|\lambda|}{\Gamma(\alpha)}\int_{0}^{\eta}\psi^{'}(s)(\psi(\eta)-\psi(s))^{\alpha-1}|f(s,x_1(s),y_1(s))-f(t,x_2(s),y_2(s))|ds\bigg]|\psi(t)-\psi(0)|\\
& \leq \frac{L_1}{\Gamma(\alpha)}\int_{0}^{t}\psi^{'}(s)(\psi(t)-\psi(s))^{\alpha-1}(|x_1(s)-x_2(s)|+|y_1(s)-y_2(s)|)ds\\
&+\frac{L_1}{|\Delta_1|}\bigg[\frac{1}{\Gamma(\alpha)}\int_{0}^{1}\psi^{'}(s)(\psi(1)-\psi(s))^{\alpha-1}(|x_1(s)-x_2(s)|+|y_1(s)-y_2(s)|)ds\\
&+\frac{|\lambda|}{\Gamma(\alpha)}\int_{0}^{\eta}\psi^{'}(s)(\psi(\eta)-\psi(s))^{\alpha-1}(|x_1(s)-x_2(s)|+|y_1(s)-y_2(s)|)ds\bigg]|\psi(t)-\psi(0)|\\
&\leq \gamma_3(\|x_1-x_2\|+\|y_1-y_2\|),
\end{align*}
which implies that
\begin{equation}\label{uniquT11}
\|\mathcal{T}_{1}(x_1,y_1)-\mathcal{T}_{1}(x_2,y_2)\|_{X}\leq \gamma_3(\|x_1-x_2\|+\|y_1-y_2\|).
\end{equation}
Similarly, we can find that
\begin{equation}\label{uniquT22}
\|\mathcal{T}_{2}(x_1,y_1)-\mathcal{T}_{2}(x_2,y_2)\|_{X}\leq \gamma_4(\|x_1-x_2\|+\|y_1-y_2\|).
\end{equation}
It follows from (\ref{uniquT11}) and (\ref{uniquT22}) that
$$\|\mathcal{T}(x_1,y_1)-\mathcal{T}(x_2,y_2)\|_{X\times Y}\leq (\gamma_3+\gamma_4)(\|x_1-x_2\|+\|y_1-y_2\|).$$
From the above inequality, we deduce that $\mathcal{T}$ is a contraction in view of the condition (\ref{uniquecondition}). Hence it follows by by Banach's fixed point theorem that there exists a unique fixed point for the operator $\mathcal{T}$, which corresponds
to a unique solution of problem (\ref{main}) on $[0,1]$. This completes the proof.
\end{proof}

\subsection{The existence result via Leray-Schauder alternative}

\begin{theorem}\label{th:existence}
Assume that (A2) hold. If $\Omega^{*}<1$, then the coupled system (\ref{main}) has at least one solution on $[0,1]$,
where $\Omega^{*}$ is given in (\ref{constants2}).
\end{theorem}
\begin{proof}
First we show that the operator $\mathcal{T}:X\times Y\to X\times Y$ is completely continuous. By the continuity of functions $f$ and $g$, the operator $\mathcal{T}$ is continuous.\\
Let $\mathcal{K}\in X\times Y$ be bounded. Then there exist constants $N_1>0,~N_2>0$ such that $|f(t,x(t),y(t))|\leq N_1$ and 
$|g(t,x(t),y(t))|\leq N_2$. Then for any $(x,y)\in\mathcal{K}$, we get
\begin{align*}
|\mathcal{T}_{1}(x,y)(t)|& \leq \frac{1}{\Gamma(\alpha)}\int_{0}^{t}\psi^{'}(s)(\psi(t)-\psi(s))^{\alpha-1}|f(s,x(s),y(s))|ds\\
&+\frac{1}{|\Delta_1|}\bigg[\frac{1}{\Gamma(\alpha)}\int_{0}^{1}\psi^{'}(s)(\psi(1)-\psi(s))^{\alpha-1}|f(s,x(s),y(s))|ds\\
&+\frac{|\lambda|}{\Gamma(\alpha)}\int_{0}^{\eta}\psi^{'}(s)(\psi(\eta)-\psi(s))^{\alpha-1}|f(s,x(s),y(s))|ds\bigg]|\psi(t)-\psi(0)|\\
& \leq \frac{N_1}{\Gamma(\alpha+1)}\Bigg[(\psi(1)-\psi(0))^{\alpha}+\frac{(|\lambda|+1)}{|\Delta_1|}(\psi(1)-\psi(0))^{\alpha+1}\Bigg],
\end{align*}
which implies that
\begin{equation}\label{b1}
\|\mathcal{T}_{1}(x,y)\|_{X}\leq \frac{N_1}{\Gamma(\alpha+1)}\Bigg[(\psi(1)-\psi(0))^{\alpha}+\frac{(|\lambda|+1)}{|\Delta_1|}(\psi(1)-\psi(0))^{\alpha+1}\Bigg].
\end{equation}
Similarly, we get
\begin{equation}\label{b2}
\|\mathcal{T}_{2}(x,y)\|_{Y}\leq \frac{N_2}{\Gamma(\beta+1)}\Bigg[(\psi(1)-\psi(0))^{\beta}+\frac{(|\mu|+1)}{|\Delta_2|}(\psi(1)-\psi(0))^{\beta+1}\Bigg].
\end{equation}
From (\ref{b1}) and (\ref{b2}), it follows that $\mathcal{T}$ is uniformly bounded.\\

Next, we shall show that the operator $\mathcal{T}$ is equicontinuous.\\

 Let $t_1,t_2\in[0,1]$ with $t_1<t_2$. Then we have\\

$|\mathcal{T}_{1}(x,y)(t_2)-\mathcal{T}_{1}(x,y)(t_1)|$
\begin{eqnarray*}
&\leq&\Bigg|\frac{1}{\Gamma(\alpha)}\int_{0}^{t_2}\psi^{'}(s)\big[(\psi(t_2)-\psi(s))^{\alpha-1}-(\psi(t_1)-\psi(s))^{\alpha-1}\big]f(s,x(s),y(s))ds\Bigg|\\
&+&\Bigg|\frac{1}{\Gamma(\alpha)}\int_{t_1}^{t_2}\psi^{'}(s)(\psi(t_1)-\psi(s))^{\alpha-1}f(s,x(s),y(s))ds\Bigg|\\
&+&\Bigg|\frac{1}{\Delta_1}\bigg[\frac{1}{\Gamma(\alpha)}\int_{0}^{1}\psi^{'}(s)(\psi(1)-\psi(s))^{\alpha-1}f(s,x(s),y(s))ds\\
&-&\frac{\lambda}{\Gamma(\alpha)}\int_{0}^{\eta}\psi^{'}(s)(\psi(\eta)-\psi(s))^{\alpha-1}f(s,x(s),y(s))ds\bigg](\psi(t_2)-\psi(t_1))\Bigg|\\
&\leq&\frac{N_1}{\Gamma(\alpha+1)}\big[(\psi(t_2)-\psi(0))^{\alpha}-(\psi(t_1)-\psi(0))^{\alpha}\big]\\
&+&\frac{N_1}{|\Delta_1|\Gamma(\alpha+1)}\big[(\psi(1)-\psi(0))^{\alpha}+|\lambda|(\psi(\eta)-\psi(0))^{\alpha}\big](\psi(t_2)-\psi(t_1)),
\end{eqnarray*}
which imply that $\|\mathcal{T}_{1}(x,y)-\mathcal{T}_{1}(x,y)\|\to 0$ independent of $( x, y )\in \mathcal{K}$ as $t_2-t_1\to 0$. Also, we get\\

$|\mathcal{T}_{2}(x,y)(t_2)-\mathcal{T}_{2}(x,y)(t_1)|$
\begin{eqnarray*}
&\leq&\frac{N_2}{\Gamma(\beta+1)}\big[(\psi(t_2)-\psi(0))^{\beta}-(\psi(t_1)-\psi(0))^{\beta}\big]\\
&+&\frac{N_2}{|\Delta_2|\Gamma(\beta+1)}\big[(\psi(1)-\psi(0))^{\beta}+|\mu|(\psi(\xi)-\psi(0))^{\beta}\big](\psi(t_2)-\psi(t_1)),
\end{eqnarray*}
which imply that $\|\mathcal{T}_{2}(x,y)-\mathcal{T}_{2}(x,y)\|\to 0$ independent of $( x, y )\in \mathcal{K}$ as $t_2-t_1\to 0$.\\
 Therefore, the operator $\mathcal{T}$ is equicontinuous. Consequently, by Arzel\`{a}-Ascoli's theorem, we deduce that the operator $\mathcal{T}$ is completely continuous.\\

Finally, we shall show that the set 
$$\mathcal{S} = \{ (x,y)\in X\times Y: (x,y)=\nu\mathcal{T}(x,y),~ 0 < \nu< 1 \}$$ is bounded.\\
Let $(x,y)\in \mathcal{S}$,  then $(x,y)=\nu\mathcal{T}(x,y)$. For any $t\in[0,1]$, we have
$$x(t)=\nu\mathcal{T}_{1}(x,y)(t),~~y(t)=\nu\mathcal{T}_{2}(x,y)(t).$$
Then we have
\begin{eqnarray*}
|x(t)|=|\nu\mathcal{T}_{1}(x,y)(t)|&\leq&\left|\mathcal{T}_{1}(x,y)(t)\right|\\
&\leq&\frac{k_0+k_1|x|+k_2|y|}{\Gamma(\alpha+1)}\Bigg[(\psi(1)-\psi(0))^{\alpha}+\frac{(|\lambda|+1)}{|\Delta_1|}(\psi(1)-\psi(0))^{\alpha+1}\Bigg],
\end{eqnarray*}
and
\begin{eqnarray*}
|y(t)|=|\nu\mathcal{T}_{2}(x,y)(t)|&\leq&\left|\mathcal{T}_{2}(x,y)(t)\right|\\
&\leq&\frac{l_0+l_1|x|+l_2|y|}{\Gamma(\beta+1)}\Bigg[(\psi(1)-\psi(0))^{\beta}+\frac{(|\mu|+1)}{|\Delta_2|}(\psi(1)-\psi(0))^{\beta+1}\Bigg].
\end{eqnarray*}
Hence, we get
\begin{eqnarray*}
\|x\|_{X}&\leq&\frac{1}{\Gamma(\alpha+1)}\Bigg[(\psi(1)-\psi(0))^{\alpha}+\frac{(|\lambda|+1)}{|\Delta_1|}(\psi(1)-\psi(0))^{\alpha+1}\Bigg](k_0+k_1\|x\|_{X}+k_2\|y\|_{Y}),\\
\|y\|_{Y}&\leq&\frac{1}{\Gamma(\beta+1)}\Bigg[(\psi(1)-\psi(0))^{\beta}+\frac{(|\mu|+1)}{|\Delta_2|}(\psi(1)-\psi(0))^{\beta+1}\Bigg](l_0+l_1\|x\|_{X}+l_2\|y\|_{Y}),
\end{eqnarray*}
which imply that
$$\|x\|_{X}+\|y\|_{Y}\leq\Omega_0+\max\{\Omega_1,\Omega_2\}\|x+y\|_{X\times Y}=\Omega_0+\Omega^{*}\|x+y\|_{X\times Y}),$$
where $\Omega_0,\Omega_1,\Omega_2$ and $\Omega^{*}$ are given in (\ref{constants2}).
Consequently, we get
\begin{equation}\label{existcond}
\|(x,y)\|_{X\times Y}\leq\frac{\Omega_0}{1-\Omega^{*}},
\end{equation}
which proves that the set $\mathcal{S}$ is bounded. Therefore, by Lemma \ref{fixed}, the operator $\mathcal{T}$ has at least one fixed point. Hence the coupled system (\ref{main}) has at least one solution on $[0,1]$. The proof is completed. 
\end{proof}

\section{Examples}
\begin{example}
Consider the following coupled system of $\psi$-Caputo fractional differential equations: 
\begin{equation}\label{ex1}
\begin{cases}
{}^{C}D_{0^{+}}^{\frac{3}{2};\psi}x(t)=\frac{e^{-3t}}{75+t}\left(\sin x(t)+|y(t)|\right)+\frac{e^{-t}}{1+t^2},~~t\in[0,1],\\
{}^{C}D_{0^{+}}^{\frac{4}{3};\psi}y(t)=\frac{1}{2t^2+100}\left(\frac{|x(t)|}{1+|x(t)|}+\sin y(t)\right)+\sin t+1,\\
x(0)=y(0)=0,\\
x(1)= x(\frac{1}{2}),~y(1)= y(\frac{1}{3}).
\end{cases}
\end{equation}
\end{example}
Here, $\alpha=\frac{3}{2},\beta=\frac{4}{3}, \eta=\frac{1}{2}, \xi=\frac{1}{3}, \lambda=\mu=1$, $f(t,x,y)=\frac{e^{-3t}}{75+t}\left(\sin x+|y|\right)+\frac{e^{-t}}{1+t^2}$ and $g(t,x,y)=\frac{1}{2t^2+100}\left(\frac{|x|}{1+|x|}+\sin y\right)+\sin t+1$.\\

Obviously, on can find that:
$$|f(t,x_1,y_1)-f(t,x_2,y_2)|\leq\frac{1}{75}(|x_1-x_2|+|y_1-y_2|),$$
$$|g(t,u_1,v_1)-g(t,u_2,v_2)|\leq\frac{1}{100}(|x_1-x_2|+|y_1-y_2|),$$
from which, we get $L_{1}=\frac{1}{75}$ and $L_{2}=\frac{1}{100}$.\\

Let us take $\psi(t)=3t^2$. Clearly, $\psi$ is an increasing function on $[0,1]$ and $\psi^{'}(t) = 6t$ is a continuous function
on $[0, 1]$.\\

Using the given data, the condition (\ref{uniquecondition}) becomes
$$\gamma_3+\gamma_4=0.1910978713+0.3633970871=0.5544949584<1.$$

Thus, all the assumptions of Theorem \ref{uniqueness} are satisfied. Hence it follows that the coupled system (\ref{ex1}) has a unique solution for  on $[0,1]$.

\begin{example}
Consider the following coupled system of $\psi$-Caputo fractional differential equations: 
\begin{equation}\label{ex2}
\begin{cases}
{}^{C}D_{0^{+}}^{\frac{3}{2};\psi}x(t)=\frac{1}{\sqrt{625+t}}\cos t+\frac{e^{-t}}{200}\sin x(t)+\frac{1}{300}\frac{y(t)|x(t)|}{1+|x(t)|},~~t\in[0,1],\\
{}^{C}D_{0^{+}}^{\frac{4}{3};\psi}y(t)=\frac{e^{-2t}}{2\sqrt{1600+t}}+\frac{1}{270}\sin x(t)+\frac{1}{3(60+t)}\sin (y(t)),\\
x(0)=y(0)=0,\\
x(1)= x(\frac{1}{2}),~y(1)= y(\frac{1}{3}).
\end{cases}
\end{equation}
\end{example}
Obviously, we get
$$|f(t,x,y)|\leq\frac{1}{25}+\frac{1}{200}\|x\|_X+\frac{1}{300}\|y\|_Y,$$
$$|g(t,x,y)|\leq\frac{1}{80}+\frac{1}{270}\|x\|_X+\frac{1}{180}\|y\|_Y.$$
Thus $k_0=\frac{1}{25},~k_1=\frac{1}{200},~k_2=\frac{1}{300},~l_0=\frac{1}{80},~l_1=\frac{1}{270},~l_2=\frac{1}{180}$.\\

Using (\ref{constants2}), we find that
$$\Omega^{*}=\max\{\Omega_1,\Omega_2\}=\max\{0.2062532154,0.5020208267\}=0.5020208267<1.$$

Therefore, the assumptions of Theorem \ref{th:existence} are satisfied. Hence, the coupled system (\ref{ex2}) has at least one solution on $[0,1]$.

\end{document}